\documentclass[12pt]{amsart}
\usepackage{amssymb,amsmath,amsthm,latexsym}
\usepackage{mathrsfs}
\usepackage{a4wide}

\newtheorem{theorem}{Theorem}[section]
\newtheorem{lemma}[theorem]{Lemma}
\newtheorem{corollary}[theorem]{Corollary}

\theoremstyle{remark}

\theoremstyle{definition}

\numberwithin{equation}{section}
\makeatother

\DeclareMathOperator{\Kdb}{{\mathbb K}}
\DeclareMathOperator{\Cdb}{{\mathbb C}}

\begin{document}

\title[]{Finite generation in $C^\ast$-algebras\\ and Hilbert $C^\ast$-modules}

%\date{October 9, 2014}
\thanks{The first author was supported by a grant from the NSF.  We are 
also grateful to the 
QOP (Quantum groups, operators and non-commutative probability) research network
and UK research council grant  EP/K019546/1 for some assistance.}

\author{David P. Blecher}
\address{Department of Mathematics, University of Houston, Houston, TX
77204-3008}
\email[David P. Blecher]{dblecher@math.uh.edu}

\author[T. Kania]{Tomasz Kania}
\address{Institute of Mathematics, Polish Academy of Sciences, \'{S}niadeckich 8, 00-956 Warszawa, Poland and Department of Mathematics and Statistics, Fylde College, Lancaster University, Lancaster LA1 4YF, United Kingdom}
\email{tomasz.marcin.kania@gmail.com}

\begin{abstract}  We characterize $C^*$-algebras and $C^*$-modules
such that every maximal right ideal (resp.\ right submodule)
is algebraically finitely generated. In particular, $C^*$-algebras satisfy the Dales--\.Zelazko conjecture.
\end{abstract}

\maketitle

\section{Introduction}  

Magajna's paper \cite{Mag} characterizing $C^*$-modules consisting of compact 
operators has been much
emulated, as is revealed by a cursory search in a citation index.   
Here we prove a complementary characterization, inspired by the recent Dales--\.Zelazko conjecture that 
if $A$ is a  unital Banach algebra all of whose maximal right ideals are algebraically finitely generated as right modules over $A$, 
then $A$ is 
finite dimensional \cite{DZ}.   Indeed the instigation of this paper was a question Dales asked 
independently of both authors, and which both authors answered
 around August 2012, as to whether this conjecture was true for $C^*$-algebras.
(He was able to answer this for special classes of  $C^*$-algebras.)   
One ingredient  of the solution is a characterization of algebraically finitely generated one-sided
ideals in $C^*$-algebras.  Although this is well known to experts (the 
algebraically
finitely generated projective modules over a $C^*$-algebra constitute  
one of the common ways to picture its K-theory, and hence are 
well understood), we could not find
it in the literature. Thus we include a direct proof due to R\o{}rdam, as well as a very
short $C^*$-module proof.  We then use this to characterize $C^*$-algebras and $C^*$-modules
such that every maximal right ideal (resp.\ right submodule)
is algebraically finitely generated.

Turning to notation and background, we denote by $A^1$ the unitization of the $C^*$-algebra $A$.  
By `projection' in this paper we mean a self-adjoint idempotent $e$ in $A$.
Then $e$ is a {\em minimal projection} in $A$ if  $e A e$ is
one dimensional (which if $A$ is a von Neumann algebra,
is equivalent to $e$ having no non-trivial proper subprojections).
 For convenience we usually work with right modules
in this paper.  It is well known that all $C^*$-algebras have an abundant supply of maximal
right ideals.   This is equivalent to saying that the bidual of $A$, $A^{\prime \prime}$, which is a von Neumann algebra, has 
an abundant supply of non-zero minimal projections (see 3.13.6 in  \cite{Ped}, or the paragraph before 
Lemma \ref{even1} below, for the correspondence
between minimal projections and maximal right ideals).
 Indeed, every right ideal is an intersection of maximal
right ideals (see 3.13.5 in  \cite{Ped}).  We will not really 
use the facts in the present paragraph though,
except for those which we prove below.
 
Before we proceed we require another piece of terminology. Let $H$ and $K$ be Hilbert spaces. A closed subspace $Z$ of $\mathscr{B}(K, H)$ is called a {\em ternary ring of operators} (or a \emph{TRO}, for short) if it is closed under the \emph{ternary product}, that is, $ZZ^*Z\subseteq Z$. Every Hilbert $C^*$-module $Z$ may be viewed as a TRO by identifying it with the ($1$-$2$)-corner of its 
linking algebra (see \emph{e.g.}\ 8.1.19 and 8.2.8 in \cite{BLM}).  Thus we will write $z^* w$ in place of $\langle z, w \rangle$ for elements in a
right  $C^*$-module $Z$.  Also, the so-called compact operators $\Kdb(Z)$ may be written as $Z Z^*$ (here and below for sets
$X, Y$ we write $XY$ for the closure of the span of products $xy$ for $x \in X, y \in Y)$.
To say that two TRO's are isomorphic as TRO's means that there is a 
linear isomorphism
between them which is a ternary morphism (that is, $T(xy^*z) = T(x)T(y)^* T(z)$).  
Hamana showed that this is equivalent to inducing
a corner-preserving $*$-isomorphism between the
Morita linking $C^*$-algebras of the TRO's; and it is also equivalent to being 
completely isometric as operator spaces (a result also contributed to by Harris, Kaup, 
Kirchberg, Ruan, and no doubt others; see \emph{e.g.}\ \cite{BLM} for references and self-contained 
proofs).   

\section{Finitely generated ideals}

The following lemma is well known to experts, although we could not find a 
reference for it. We shall present a direct self-contained proof that we are grateful to
Mikael R\o{}rdam for having communicated to us.  Of course there are many other proofs,
including the one in the next Remark.

\begin{lemma} \label{fgc}  Every algebraically finitely generated closed left (resp.\ right) ideal of a $C^*$-algebra $A$  is actually singly generated, and equals $Ap$ (resp.\ $pA$) for a projection $p \in
A$.  \end{lemma} 

\begin{proof} 
We consider first the case where $A$ is unital. Let $J\subseteq A$ be a closed left ideal. We shall use the following
  general fact: for each positive element $a\in J$ and
  each continuous function \mbox{$f\colon[0,\infty)\to\mathbb{R}$} with
  $f(0)=0$, we have $f(a)\in J$. (This follows by
  approximating~$f$ uniformly on the spectrum of~$a$ by real
  polynomials vanishing at $0$.)

  Suppose that~$J$ is generated as a left ideal by the
  elements $a_1,\ldots,a_n$ for some $n\in\mathbb{N}$, and set $b = a_1^*
  a_1 + \cdots + a_n^* a_n\in J$. Then~$b^{1/4}$ belongs
  to~$J$ by the above fact, so that $b^{1/4}= c_1a_1+ \cdots
  +c_na_n$ for some $c_1,\dots,c_n\in A$. We may suppose
  that~$J$ is non-zero, which implies that $b$ is non-zero,
  and consequently $c_1,\dots,c_n$ are not all zero. Since $x^* x + y^* y - (x^* y + y^* x) =
(x-y)^* (x-y) \geqslant 0$
for any $x, y \in A$,  we deduce that
$$a_j^*
  c_j^* c_ka_k + a_k^* c_k^* c_ja_j 
  \leqslant a_j^* c_j^* c_ja_j + a_k^*
  c_k^* c_ka_k$$ for $j\ne k$.  Hence
  \[ b^{1/2} = (b^{1/4})^* b^{1/4} = \sum_{j,k=1}^n a_j^*
  c_j^* c_ka_k\leqslant n\sum_{j=1}^na_j^* c_j^* c_ja_j\leqslant nK
  \sum_{j=1}^na_j^* a_j = nKb, \] where $K = \max_{1\leqslant j\leqslant
    n}\|c_j\|^2 > 0$. By elementary spectral calculus, this implies
  that the spectrum of~$b$ is contained in the set
  $\{0\}\cup[(nK)^{-2},\infty)$, so that we can take a continuous
  function $f\colon[0,\infty)\to [0,1]$ such that $f(0)=0$ and
  $f(t)=1$ for each $t\geqslant (nK)^{-2}$.  Then $p=f(b)$ is a projection such that $pb = bp = b$, and $p$ belongs
  to~$J$ by the fact stated above. In particular we have 
\[0
  = (1-p)b(1-p) = \sum_{j=1}^n\big(a_j(1-p)\big)^* a_j(1-p),\] which
  implies that $a_j = a_jp\in Ap$ for each
  $j\in\{1,\ldots,n\}$.
    Hence $J =
  Ap$, and the result follows.

Let us now consider the case where $A$ is non-unital. Let $J$ be a closed, finitely generated left ideal of $A$. Then $J$ is finitely generated when regarded as a 
left ideal of $A^1$. 
Let $p$ be a projection in $A^1$ such that $J =A^1p = \{ x \in A^1 : xp = x \}$. Then $p =  1p \in J \subset A$, so $$J =  \{ x \in A : xp = x \} =Ap.$$  

The right-ideal case is similar or follows by symmetry by considering the opposite $C^*$-algebra.
\end{proof}

\noindent
{\bf Remark.}  Lemma  \ref{fgc} also follows from a well-known $C^*$-module 
`generalization' of it, 
which is a basic result 
in the theory of  Hilbert $C^*$-modules (see \emph{e.g.}\ 
p.\ 255--257 in \cite{W-O} or the proof of 8.1.27 in \cite{BLM}).  Namely, a 
right $C^*$-module $Z$ over $A$ is  algebraically finitely generated over 
$A$ iff there are
finitely many $z_k \in Z$ with $z = \sum_k z_k z_k^* z$ for all $z \in Z$.  Note that this immediately implies  Lemma \ref{fgc} by taking $Z$ to be the 
right ideal of $A$ in Lemma \ref{fgc}: in this case if $e = \sum_k z_k z_k^*$, which is in $Z$, then $e^2 = e$
and $e \geqslant 0$.  So $e$ is a projection in the right ideal, and now it is
easy to see that this right ideal equals $eA$.

\bigskip

If $K$ is a maximal right ideal of $A$ 
then  $e$, the complement of the support projection of $K$, is a minimal projection
in $A^{\prime \prime}$.  
This is well known (see 3.13.6 in  \cite{Ped}), but 
here is a simple argument for this.   We recall that 
the support projection of $K$ is the smallest projection $p \in A^{\prime \prime}$ with 
$px = x$ for all $x \in K$.   Thus $e$ is the largest projection in $A^{\prime \prime}$ with
$ex = 0$ for all $x \in K$.  We will assume for simplicity
that  $e \in A$, which will 
be the case for us in
 Corollary \ref{ocha} below, but the general case is 
very similar (but uses modifications
of some steps below using Cohen factorization and 
`second dual techniques' valid in any Arens regular Banach algebra,
and one should replace $eAe$ and $eA$ below by $\{ a \in A\colon a = eae \}$ and
$\{ a \in A\colon a = ea \}$). We 
will use only the well-known fact that every non-trivial $C^*$-algebra has a proper non-zero closed right ideal, \emph{e.g.}\ the right kernel of any non-faithful state. 
If $e$ is not minimal, that is if $A_e = eAe$ is not one dimensional, then $A_e$ has a proper closed non-zero right ideal $I$, and $I = I A_e$ as usual.  Then $W = I A$ is a closed right ideal of $A$.  Note that  
 $W \neq eA$ since $\{ w \in W : w e = w \} \subset I \neq A_e$.  On the  other hand, 
$K + W = A$ by maximality of $K$ (note $K \cap W \subset (1-e) A \cap eA = \{0\}$).   Thus 
$e A =  e(K+W) = W$.   This contradiction shows that 
$e$ is a minimal projection. 

\begin{lemma}  \label{even1}  A   $C^*$-algebra $A$  is  unital
if even one maximal right ideal is algebraically finitely generated over $A$.
 \end{lemma} 

\begin{proof}  As we said above, a maximal right ideal of $A$ has a support projection whose complement is a minimal projection $q \in A^{\prime \prime}$. On the other hand, if $J$ is an algebraically finitely generated right ideal
then by Lemma \ref{fgc} the support projection of $J$  is in $J$.  Thus if $J$ is an
 algebraically finitely generated  right ideal which is a maximal right ideal, then $1-q \in J \subset A$ for
a  non-zero minimal projection $q$  in  $A^{\prime \prime}.$
Hence $q = 1 - (1-q)$ belongs to $M(A)$, the multiplier algebra of $A$, and of course 
$q A q \neq \{0\}$ since $q \neq 0$.  Therefore $\{0\} \neq qAq = \Cdb q \subset A$, and so $q$ and $1 = (1-q) + q$ are 
in $A$.   So $A$  is unital.
\end{proof}  

\begin{corollary} \label{ocha}  A   $C^*$-algebra $A$  is  finite dimensional iff every  maximal right ideal is algebraically finitely generated over $A$.  
\end{corollary}

\begin{proof}  For the non-obvious direction, by Lemma~\ref{even1} we may suppose that $A$ is unital. 
Let $J$ be the right ideal generated by all the minimal projections in $A$.  If $J
 \neq A$ let $K$ be a maximal (proper) right ideal of $A$ containing $J$.  
The support
projection of $K$ is in $A$ by Lemma~\ref {fgc}, hence
its complement $e$ is in $A$ too.
As we proved above  Lemma~\ref{even1}, $e$
is a minimal projection, and we obtain the contradiction
$e \in J \subseteq K =  (1-e)A$.  So $A = J$, and therefore 
$1 = \sum_{k=1}^n \, e_k a_k = \sum_{k=1}^n \ a_k^* e_k$ for 
minimal projections $e_k$, and 
some $a_k \in A$.    It is well known from pure algebra that dim$(e A f) \leqslant 1$ 
for minimal $e, f \in A$ (a
quick proof in our case where these are 
projections: if $v = ea f \neq 0$ then $v^* v$ is a positive scalar multiple of $f$, so that 
left multiplication by $v^*$ is an isomorphism $eAf \cong fAf$).  From these facts it is clear that 
$A  = \sum_{j,k=1}^n \, e_j A e_k$ is finite dimensional.   
\end{proof}
\noindent
{\bf Remark.}   Although 
we have chosen to give a selfcontained
 $C^*$-algebraic argument in the last proof,  there are more algebraic 
arguments available that even allow one to generalize some of the above.
For example, note that the hypothesis in the last result together with a result of the type of 
Lemma \ref{fgc}, implies that every maximal right ideal is a (module) direct summand.    
But the latter implies finite dimensionality.   Indeed the elementary 
 argument in the  lines after Proposition 5.10
in \cite{Rey} (which is a slight variant of our argument 
in the last proof) shows that any ring $A$ whose maximal right ideals are (module) direct summands,
equals its socle.  Hence $A$ is semisimple in the ring-theoretic sense, and one can apply the Wedderburn-Artin theorem.  We thank Manuel Reyes for the last reference.  So  if $A$ is in addition  a Banach algebra over $\Cdb$ it is now clear 
that it is  finite dimensional.   Similarly,
one obtains the well known fact that a unital Banach algebra 
with dense socle (and hence equals its socle) is  finite dimensional.
   This is also related to the theory of  {\em modular annihilator algebras}
(see  \emph{e.g.}\  8.4.14 in \cite{Pal}, and its proof).   

\begin{corollary}  \label{ochamin}  A   unital $C^*$-algebra $A$  is  finite dimensional iff $A$ contains all 
 minimal projections in $A^{\prime \prime}$.
\end{corollary} \begin{proof}  If $A$ contains all such projections
and $J$ is a maximal right ideal of $A$, then
the support projection $p$ of $J$ is in $A$ (since its complement is a minimal projection).
So $J = pA$.  The result now follows from 
Corollary \ref{ocha}.   
\end{proof}  
\noindent
{\bf Remark.}   One might ask which of the results above extend to the class
of
not necessarily self-adjoint algebras of operators
on a Hilbert space (resp.\ to 
classes of Banach algebras).  In \cite{BRIII,BZ}  there are variants
of one or two of the facts above for closed right ideals with a contractive (resp.\
 `real-positive' left approximate identity).  For example, comparing with Lemma \ref{fgc},
such right ideals which are algebraically finitely generated as 
right modules over the algebra $A$,
are precisely the right ideals  of the form $eA$ for a projection $e$ (resp.\
a `real-positive' idempotent) in the algebra 
 (see \cite[Corollary 2.13]{BRIII}
and \cite[Corollary 4.7]{BZ}; in the latter reference it is also assumed
that $A$ has a contractive approximate identity
but probably this is not necessary).   Comparing with 
Lemma \ref{even1}, and following its
proof, one sees that $A$  is a unital operator algebra say, 
which possesses even one such ideal 
which is algebraically finitely generated over $A$, and which is maximal in the 
sense that the complement $e$ of its support projection is minimal in the sense 
that  $e A^{\prime \prime} e$ is one dimensional.   However even if $A$ is unital, it 
need not have any right ideals of this type at all.  Thus our techniques above towards
 the Dales--\.Zelazko conjecture break down in this case, although our method suggests
that the way to proceed may be via the socle of $A$.   
      
\section{A $C^*$-module generalization}  

We now show that $C^*$-modules of the form  $\bigoplus_{k  = 1}^m \,  \mathscr{B}(\Cdb^{n_k},H_k)$ (that is,   
direct sums of rectangular matrix blocks with the length of the rows in each block allowed to be infinite), 
are the `only'  right $C^*$-modules $Z$ such that  every maximal right submodule of $Z$ is  algebraically finitely generated.

\begin{theorem} \label{mfgcs}   Let $Z$ be a right $C^*$-module.
Then every maximal right submodule of $Z$ is  algebraically finitely generated  iff  there are positive  integers $m, n_1, \cdots, n_m$, and Hilbert spaces $H_k$, such that  $Z \cong \bigoplus_{k  = 1}^m \,  \mathscr{B}(\Cdb^{n_k},H_k)$ as TRO's.   
\end{theorem}   \begin{proof}  ($\Rightarrow$) \   Suppose that  $Z$ is a right $C^*$-module over a $C^*$-algebra $B$, and that every maximal right submodule of $Z$ is  algebraically finitely generated over $B$.  Then  every maximal right submodule $W$ of $Z$ is  algebraically finitely generated over $Z^* Z$
(since $W$ is a non-degenerate $Z^* Z$-module and hence any $w \in W$ may be written as $w = w' c$ for 
$w' \in W, c \in Z^* Z$ by Cohen's factorization theorem.  Hence $w b = w' (cb)$ with $cb \in Z^* Z$, for $b \in B$).   So we may assume that $B = Z^* Z$.  

We will be using the simple relationship between right submodules of $Z$ and right ideals of $Z Z^*$ perhaps first noticed by Brown \cite{Br}. If $J$ is a maximal right ideal of $A = \Kdb(Z) = Z Z^*$, then $J Z$ is a  right submodule of $Z$.  If   $J Z = Z$ then 
$$J = J A = J Z Z^* = Z Z^* =  A,$$ a contradiction.  So $JZ$ is a proper  right submodule of $Z$.    If
$W$ is a proper closed right submodule of $Z$ containing $JZ$, then $W Z^*$ is a right ideal of $\Kdb(Z)$ and
it contains $J A= J$.  If $W Z^* = A$, then $W = W Z^* Z = A Z = Z$, a contradiction.  Hence $W Z^* = 
J$, so that  $W = W Z^* Z = J Z$.  Thus $JZ$ is a maximal right  submodule of $Z$, and hence
$JZ$ is finitely generated over $Z^*Z$.  By the well-known argument/fact in the remark after Lemma \ref{fgc} 
above, $JZ$ has generators $z_1, \ldots , z_n$ with $$\sum_{k = 1}^n \, z_k z_k^* a z = az$$ for all  $a \in J, z \in Z$.   Hence 
$e a = a$ for all  $a \in J$ where $e = \sum_{k = 1}^n \, z_k z_k^* \in J$.  Clearly $J = e Z Z^*$.  By Lemma \ref{even1} we see that $Z Z^*$ is unital,
and by 
Corollary \ref{ocha}  we have that  $Z Z^*$ is a finite dimensional $C^*$-algebra, hence 
 $Z Z^* \cong \bigoplus_{k  = 1}^m \, M_{n_k}$ $*$-isomorphically.  Now we are in well-known territory,
indeed Hilbert $C^*$-modules over $C^*$-algebras of compact operators are completely understood. 
For example, by basic Morita equivalence 
(as in \emph{e.g.}\ the proof on pp.\ 851--852 in \cite{Mag}, or p.\ 2125 of \cite{Sch}) we have $Z^* Z \cong \bigoplus_{k  = 1}^m \, \Kdb(H_k)$, and $Z \cong  \bigoplus_{k  = 1}^m \,  \mathscr{B}(\Cdb^{n_k},H_k)$, for
 Hilbert spaces $H_k$.  (The cited papers do not explicitly use the term 
`ternary morphism', but it is clear
that their morphisms are such.)  
 
 ($\Leftarrow$)   This is the easy direction.  Indeed, if $Z =  \bigoplus_{k  = 1}^m \,  \mathscr{B}(\Cdb^{n_k},H_k)$ then every 
right $Z^*Z$-submodule $W$ is finitely generated over $Z^* Z$ (since $W W^*$ is finite dimensional, hence unital).  
And clearly this property is preserved by ternary isomorphisms.
\end{proof}  
\noindent
{\bf Remark.}  All  right $C^*$-modules (which are not Hilbert spaces) have an abundant supply of maximal
right submodules (one can see this for example from the paragraph before Lemma \ref{even1}, and the correspondence in the proof of
Theorem \ref{mfgcs}).  Indeed, every right submodule is an intersection of maximal
right submodules.

\bigskip
\noindent {\bf Closing remark.}
Let $\kappa$ be a cardinal number. We will say that a right module $V$ over $A$ is \emph{algebraically} $\kappa$-\emph{generated} if there is a set $\{v_\alpha\colon \alpha<\kappa\}$ in $V$ with cardinality $\kappa$ such that
every element in $A$ is a finite sum $\sum_{k=1}^n \, v_{\alpha_k} a_k$ for some 
$a_k \in A^1$ and $\alpha_1, \ldots , \alpha_n<\kappa$. We call algebraically $\aleph_0$-generated modules \emph{algebraically countably generated}.  One might ask if  
`algebraically finitely generated' could be replaced by `algebraically
countably generated' or `algebraically $\kappa$-generated' for some uncountable cardinal $\kappa$ in all of the results in our paper.  In fact this is automatic in the countable case: 
It is proved in \cite{Boudi} that a right ideal of a Banach algebra
is closed if its closure is algebraically 
countably generated in this sense.  The proof in \cite{Boudi} works
for modules too; thus a right submodule of a Banach module over $A$
is closed if its closure is algebraically 
countably generated.   Then as in  \cite[Corollary 1.6]{DZ},
closed algebraically
countably generated right submodules of a Banach module over $A$
are finitely generated.  One can even go one step further using some set theory related 
to Martin's axiom.  We shall use the
 so-called \emph{pseudo-intersection number} $\mathfrak{p}$, a certain cardinal.
That is, $\mathfrak{p}$ is the minimal cardinality of a family $(U_\alpha)_{\alpha<\lambda}$ of open dense subsets of $\mathbb{R}$ such that $\bigcap_{\alpha<\lambda}U_\alpha$ is not dense in $\mathbb{R}$.

\begin{corollary} \label{cgen}  A closed algebraically
countably generated right submodule of a Banach module over $A$
is finitely generated.  Moreover, if a closed algebraically $\kappa$-generated right submodule of a Banach module is separable, where $\kappa<\mathfrak{p}$, then it is finitely generated.
\end{corollary}

\begin{proof}  The `countably generated' case is just as in the proof of \cite[Corollary 1.6]{DZ},
but using the module version of Boudi's result discussed above.
In the other case, 
let $G$ be a set of algebraic generators for a closed submodule $I$,
 with $|G|< \mathfrak{p}$.  Then
the family of all finite subsets of $G$ has the same cardinality.  
For cardinals $< \mathfrak{p}$, there is a 
generalization of Baire's category theorem valid in
 separable metric spaces; see \emph{e.g.}\ \cite[Corollary 22C]{Frem}.
We proceed similarly to the proof of \cite[Corollary 1.6]{DZ}, but apply this generalized   
Baire principle to the  union of the closed submodules generated by finite subsets of $G$, 
to see that one such submodule equals $I$.  Finally, apply the module version of Boudi's result discussed above.  
\end{proof}

Let us note that the separability assumption in Corollary~\ref{cgen} cannot be dropped. Indeed, let $A=C[0,\omega_1]$, that is, $A$ is the commutative $C^*$-algebra of all continuous functions on the ordinal interval $[0,\omega_1]$. Let $I$ be the ideal of $A$ consisting of functions which vanish at $\omega_1$. (As a Banach space, $I$ is clearly non-separable.) Each function $f$ in $I$ has countable support $\mbox{supp}\, f$, since
 continuous functions on $[0,\omega_1]$ are eventually constant. Let $f\in I$. We can then write $f = f\cdot \mathbf{1}_{[0,\alpha]}$, where $\mathbf{1}_{[0,\alpha]}$ is the characteristic function of the ordinal interval $[0,\alpha]$ and $\alpha = \sup \mbox{supp}\, f$. 
Since $\alpha$ is countable, we have
 $\mathbf{1}_{[0,\alpha]} \in I$.  Thus 
 $I$ is not finitely generated, but is algebraically
 $\aleph_1$-generated (regardless of whether $\aleph_1<\mathfrak{p}$ or not).

\end{document}